\documentclass{article}

\usepackage{amsmath,amsfonts,amsthm}

\newtheorem{theorem}{Theorem}
\newtheorem{corollary}[theorem]{Corollary}
\newtheorem{lemma}[theorem]{Lemma}

\begin{document}

\title{Lyapunov-type inequality for a fractional boundary value problem 
with natural conditions\thanks{This is a preprint of a paper whose final 
and definite form is with \emph{SeMA Journal}, ISSN: 2254-3902 (Print) 2281-7875 (Online). 
Submitted 05-March-2017; Article revised 04-Apr-2017; Article accepted for publication 24-Apr-2017.}}

\author{Assia Guezane-Lakoud$^{1}$\\ 
{\small \texttt{a\_guezane@yahoo.fr}}
\and	
Rabah Khaldi$^{1}$\\ 
{\small \texttt{rkhadi@yahoo.fr}}
\and 
Delfim F. M. Torres$^{2}$\\
{\small \texttt{delfim@ua.pt}}}

\date{$^{1}$Laboratory of Advanced Materials,\\
Department of Mathematics,\\
Badji Mokhtar-Annaba University,\\
P.O. Box 12, 23000 Annaba, Algeria\\[0.3cm]
$^2$Center for Research and Development\\
in Mathematics and Applications (CIDMA),\\
Department of Mathematics, University of Aveiro,\\
3810-193 Aveiro, Portugal}

\maketitle


\begin{abstract}
We derive a new Lyapunov type inequality for a boundary value
problem involving both left Riemann--Liouville and right Caputo 
fractional derivatives in presence of natural conditions.
Application to the corresponding eigenvalue problem is also discussed.

\medskip

\noindent \textbf{Keywords:} fractional calculus, 
Lyapunov inequality, eigenvalue problem.

\medskip

\noindent \textbf{MSC 2010:} 26A33, 26D15, 34A08.
\end{abstract}


\section{Introduction}

Lyapunov's inequality is a useful tool in the study of spectral 
properties of ordinary differential equations \cite{MR3443424,14}. 
The classical Lyapunov inequality is given in the following theorem.

\begin{theorem}[See \cite{10,12}] 
If the boundary value problem
\begin{equation}
\label{1.1}
\begin{gathered}
-u^{\prime \prime }\left( t\right) 
=q\left( t\right) u\left( t\right), \quad a<t<b,\\
u\left( a\right) =u\left( b\right) =0
\end{gathered}
\end{equation}
has a nontrivial continuous solution, where 
$q$ is a real and continuous function, then
\begin{equation}
\label{1.2}
\int_{a}^{b}\left\vert q\left( t\right) \right\vert dt
\geq \frac{4}{b-a}. 
\end{equation}
Furthermore, the constant 4 in \eqref{1.2} is sharp. 
\end{theorem}

Many authors have extended the Lyapunov inequality 
\eqref{1.2} \cite{MR3443424,11,14,15}. Here we are interested
in generalizations of \eqref{1.2} that are associated to
a fractional differential equation, where the second order derivative 
in \eqref{1.1} is substituted by some fractional operator \cite{1,4,7,8,10,9}.

Recently, in 2016, Ferreira obtained Lyapunov type inequalities 
for Caputo or Riemann--Liouville sequential fractional differential 
equations with Dirichlet boundary conditions \cite{8}.
In 2017, Agarwal and \"Ozbekler obtained Lyapunov type inequalities 
for mixed nonlinear Riemann--Liouville fractional differential equations 
with a forcing term and Dirichlet boundary conditions \cite{1}. 
Here, we prove a new Lyapunov type inequality for a sequential
fractional boundary value problem involving both Riemann--Liouville 
and Caputo fractional derivatives:
\begin{equation}
\label{1.3}
-^{C}D_{b^{-}}^{\alpha }D_{a^{+}}^{\beta }u\left( t\right) 
+q(t)u\left(t\right) =0, \quad a<t<b,  
\end{equation}
where $0<\alpha ,\beta \leq 1$, $1<\alpha +\beta \leq 2$, 
$^{C}D_{b^{-}}^{\alpha }$ denotes the right Caputo derivative, 
$D_{a^{+}}^{\beta}$ denotes the left Riemann--Liouville derivative, 
$u$ is the unknown function, and $q$ is continuous on $\left[ a,b\right]$. 
Such problems, with both left and right fractional derivatives, arise 
in the study of Euler--Lagrange equations for fractional problems 
of the calculus of variations \cite{MR3443073,MR3331286,MR2610543}. 
We consider the fractional differential equation \eqref{1.3} 
with natural boundary conditions \cite{2,3}: 
\begin{equation}
\label{1.4}
u\left( a\right) =D_{a^{+}}^{\beta }u\left( b\right) =0.
\end{equation}
To the best of our knowledge, this is the first work 
to give a Lyapunov type inequality (see Theorem~\ref{thm:mr}) 
for mixed right Caputo and left Riemann--Liouville fractional 
differential equations. The result is important because, in many applications, 
the natural boundary conditions have a physical interpretation. For instance, 
fractional variational problems require imposition of natural boundary 
conditions that the optimum solution must satisfy \cite{MR3443073}. 
Moreover, for boundary value problems, when sufficient kinematic conditions 
are not specified, the natural boundary conditions are necessary to solve 
the problem analytically \cite{MR2610543}. Hence, natural boundary conditions 
are necessary to solve a fractional boundary value problem, and the 
fractional problem with natural boundary conditions is not obvious 
neither trivial. This is the case of the problem studied in our paper, where
condition \eqref{1.4} is imposed naturally, because we have both a right Caputo 
derivative and a left Riemann--Liouville derivative in equation \eqref{1.3}. 

The paper is organized as follows.
In Section~\ref{sec2}, we briefly recall
the necessary concepts and results from fractional calculus.
Our results are then formulated and proved in Section~\ref{sec3}.
We end with Section~\ref{sec4}, where an example of application 
to a fractional eigenvalue problem is given, and 
Section~\ref{sec5} of conclusion.


\section{Preliminaries}
\label{sec2}

We recall here the essential definitions on fractional calculus. 
For details on the subject we refer the reader to \cite{11,15}.
Let $p>0$. Then the left and right Riemann--Liouville fractional 
integral of a function $g$ are defined respectively by
\begin{eqnarray*}
I_{a^{+}}^{p}g(t) 
&=&\frac{1}{\Gamma \left( p\right) }\int_{a}^{t}
\frac{g(s)}{(t-s)^{1-p}}ds,\\
I_{b^{-}}^{p}g(t) &=&\frac{1}{\Gamma\left( p\right) }
\int_{t}^{b}\frac{g(s)}{(s-t)^{1-p}}ds.
\end{eqnarray*}
The left Riemann--Liouville fractional derivative and the right Caputo
fractional derivative of order\ $p>0$ of a function $g$ are
\begin{eqnarray*}
D_{a^{+}}^{p}g(t) 
&=&\frac{d^{n}}{dt^{n}}\left( I_{a^{+}}^{n-p}g\right) (t),\\
^{C}D_{b^{-}}^{p}g(t) 
&=&\left( -1\right) ^{n}I_{b^{-}}^{n-p}g^{\left(n\right) }(t),
\end{eqnarray*}
respectively, where $p \in (n-1, n)$. With respect to the properties 
of Riemann--Liouville and Caputo fractional operators, 
we mention the following. Let $p \in (n-1, n)$ 
and $f\in L_{1}\left[ a,b\right]$. Then,
\begin{enumerate}
\item $I_{a^{+}}^{p}D_{a^{+}}^{p}f\left( t\right) 
=f\left( t\right)-\sum_{i=1}^{n}c_{i}\left( t-a\right) ^{p-i}$;

\item $I_{b^{-}}^{p}$$^{C}D_{b^{-}}^{p}f\left( t\right) 
=f\left( t\right)-\sum_{k=0}^{n-1}
\frac{\left( -1\right)^{k}f^{\left( k\right)}\left(b\right)}{k!}\left( b-t\right)^{k}$.
\end{enumerate}


\section{Lyapunov-type inequality}
\label{sec3}

We begin by transforming problem \eqref{1.3}--\eqref{1.4} 
into an equivalent integral equation.

\begin{lemma}
Assume that $0<\alpha ,\beta \leq 1$ and $1<\alpha +\beta \leq 2$.
Function $u$ is a solution to the boundary value problem 
\eqref{1.3}--\eqref{1.4} if and only if $u$ satisfies 
the integral equation 
\begin{equation*}
u\left( t\right) =\int_{a}^{b}G(t,r)q(r)u(r)dr,  
\end{equation*}
where
\begin{equation}
\label{2.2}
G(t,r)=\left\{ 
\begin{array}{c}
\frac{1}{\Gamma \left( \alpha \right) 
\Gamma \left( \beta \right) } \displaystyle 
\int_{a}^{r}\left( t-s\right) ^{\beta -1}\left( r-s\right) ^{\alpha-1}ds, 
\quad a\leq r\leq t\leq b, \\[0.3cm] 
\frac{1}{\Gamma \left( \alpha \right) \Gamma \left( \beta \right) }
\displaystyle \int_{a}^{t}\left( t-s\right)^{\beta -1}\left(r
-s\right)^{\alpha-1}ds, \quad a\leq t\leq r\leq b.
\end{array}
\right.   
\end{equation}
\end{lemma}

\begin{proof}
Applying the properties of Caputo and Riemann--Liouville fractional
derivatives and the boundary conditions \eqref{1.4}, then using the Fubini
theorem, we obtain 
\begin{eqnarray*}
u\left( t\right)  
&=&I_{a^{+}}^{\beta }I_{b^{-}}^{\alpha }q(t)u(t)\\
&=&\frac{1}{\Gamma \left( \alpha \right) \Gamma \left( \beta \right)}
\int_{a}^{t}\left( t-s\right) ^{\beta -1}\left( \int_{s}^{b}\left(
r-s\right) ^{\alpha -1}q(r)u(r)dr\right) ds \\
&=&\frac{1}{\Gamma \left( \alpha \right) \Gamma \left( \beta \right)}
\int_{a}^{t}\left( \int_{a}^{r}\left( t-s\right) ^{\beta -1}\left(
r-s\right) ^{\alpha -1}ds\right) q(r)u(r)dr \\
&&\quad +\frac{1}{\Gamma \left( \alpha \right) \Gamma \left( \beta \right)}
\int_{t}^{b}\left( \int_{a}^{t}\left( t-s\right)^{\beta -1}\left(
r-s\right) ^{\alpha -1}ds\right) q(r)u(r)dr,
\end{eqnarray*}
from which the intended result follows.
\end{proof}

We now prove some properties of the Green function \eqref{2.2}.

\begin{lemma}
\label{lemma3}
Assume that $0<\alpha, \beta \leq 1$ and $1<\alpha +\beta \leq 2$. Then 
the Green function $G$ defined by \eqref{2.2} satisfies the following properties:
\begin{enumerate}
\item $G(t,r)\geq 0$ for all $a\leq r\leq t\leq b$;

\item $\underset{t\in \lbrack a,b]}{\max}G(t,r)=G(r,r)$ 
for all $r\in \lbrack a,b]$;

\item $\underset{r\in \lbrack a,b]}{\max }G(r,r)
=\frac{(b-a)^{\alpha +\beta -1}}{\left( \alpha +\beta -1\right) 
\Gamma \left( \alpha \right) \Gamma \left(\beta \right) }$.
\end{enumerate}
\end{lemma}

\begin{proof}
Obviously, $G(t,r)\geq 0$ for $t,r\in (a,b)$. Set 
\begin{eqnarray*}
g_{1}(t,r) 
&=&\frac{1}{\Gamma \left( \alpha \right) \Gamma \left( \beta\right)}
\int_{a}^{r}\left( t-s\right) ^{\beta -1}\left( r-s\right) ^{\alpha-1}ds,
\quad a\leq r\leq t\leq b,\\
g_{2}(t,r) 
&=&\frac{1}{\Gamma \left( \alpha \right) \Gamma \left( \beta\right)}\left( 
\int_{a}^{t}\left( t-s\right) ^{\beta -1}\left( r-s\right)^{\alpha -1}ds\right),
\quad a\leq t\leq r\leq b.
\end{eqnarray*}
For $r\leq t$, we have
\begin{equation}
\label{2.3}
g_{1}(t,r)\leq \frac{1}{\Gamma \left( \alpha \right) 
\Gamma \left(\beta\right) }\int_{a}^{r}\left( 
r-s\right)^{\beta -1}\left( r-s\right)^{\alpha-1}ds=G(r,r).  
\end{equation}
Similarly, if $t\leq r$, then 
\begin{equation}
\label{2.4}
\begin{split}
g_{2}(t,r)&\leq \frac{1}{\Gamma \left( \alpha \right) \Gamma \left( \beta
\right) }\int_{a}^{t}\left( t-s\right) ^{\beta -1}\left( t-s\right)^{\alpha-1}ds\\ 
&=\frac{(t-a)^{\alpha +\beta -1}}{\left( \alpha +\beta -1\right) \Gamma
\left( \alpha \right) \Gamma \left( \beta \right) }\\
&\leq \frac{(r-a)^{\alpha+\beta -1}}{\left( \alpha 
+\beta -1\right) \Gamma \left( \alpha \right)
\Gamma \left( \beta \right) }\\
&=G(r,r).
\end{split}
\end{equation}
Thus, from \eqref{2.3} and \eqref{2.4}, we get $\underset{t\in \lbrack a,b]}{\max}
G(t,r)=G(r,r)$ for all $r\in \lbrack a,b]$. Since $G(r,r)$ is increasing,
we obtain that
\begin{equation*}
\underset{r\in \lbrack a,b]}{\max }G(r,r)
=\frac{(b-a)^{\alpha +\beta -1}}{\left( \alpha +\beta -1\right) 
\Gamma \left( \alpha \right) \Gamma \left(\beta \right) }.
\end{equation*}
The proof is complete.
\end{proof}

Now we are ready to give the Lyapunov type inequality for problem
\eqref{1.3}--\eqref{1.4}.

\begin{theorem}
\label{thm:mr}
Assume that $0<\alpha ,\beta \leq 1$ and $1<\alpha +\beta \leq 2$. 
If the fractional boundary value problem \eqref{1.3}--\eqref{1.4} 
has a nontrivial continuous solution, then
\begin{equation}
\label{2.5}
\int_{a}^{b}\left\vert q\left( r\right) \right\vert dr\geq \frac{\left(
\alpha +\beta -1\right) \Gamma \left( \alpha \right) \Gamma \left( \beta
\right) }{(b-a)^{\alpha +\beta -1}}.  
\end{equation}
Furthermore, the inequality \eqref{2.5} is sharp. 
\end{theorem}

\begin{proof}
From Lemma~\ref{lemma3}, we have
\begin{eqnarray*}
\left\vert u\left( t\right) \right\vert &\leq &\int_{a}^{b}G(t,r)\left\vert
q\left( r\right) \right\vert \left\vert u\left( r\right) \right\vert dr \\
&\leq &\int_{a}^{b}G(r,r)\left\vert q\left( r\right) \right\vert \left\vert
u\left( r\right) \right\vert dr \\
&\leq &\frac{(b-a)^{\alpha +\beta -1}\left\Vert u\right\Vert }{\left( \alpha
+\beta -1\right) \Gamma \left( \alpha \right) \Gamma \left( \beta \right) }
\int_{a}^{b}\left\vert q\left( r\right) \right\vert dr,
\end{eqnarray*}
where $\left\vert \left\vert u\right\vert \right\vert =\underset{t\in \left[
a,b\right] }{\max }\left\vert u\left( t\right) \right\vert .$ Consequently,
\begin{equation*}
\left\Vert u\right\Vert \leq \frac{(b-a)^{\alpha +\beta -1}\left\Vert
u\right\Vert }{\left( \alpha +\beta -1\right) \Gamma \left( \alpha \right)
\Gamma \left( \beta \right) }\int_{a}^{b}\left\vert q\left( r\right)
\right\vert dr.
\end{equation*}
Thus, inequality \eqref{2.5} follows.
\end{proof}

Next we give a Lyapunov type inequality in the case $\alpha =\beta =1$.

\begin{corollary}
If the boundary value problem 
\begin{gather*}
u^{\prime \prime }\left( t\right) +q\left( t\right) u\left( t\right) =0\\
u\left( a\right) =0=u^{\prime }\left( b\right)
\end{gather*}
has a nontrivial continuous solution, then the Lyapunov inequality 
\begin{equation*}
\int_{a}^{b}\left\vert q\left( r\right) \right\vert dr\geq \frac{1}{(b-a)}
\end{equation*}
holds.
\end{corollary}


\section{Application to a fractional eigenvalue problem}
\label{sec4}

We end with an application of the Lyapunov-type inequality \eqref{2.5} 
to a fractional eigenvalue problem generated 
by the fractional differential equation
\begin{equation}
\label{2.7}
^{C}D_{b^{-}}^{\alpha }D_{a^{+}}^{\beta }u\left( t\right) 
=\lambda u\left(t\right),
\quad a<t<b, \quad \lambda \in \mathbb{R},  
\end{equation}
subject to the boundary conditions \eqref{1.4}.

\begin{corollary}
Assume that $0<\alpha, \beta \leq 1$ and $1<\alpha +\beta \leq 2$. 
If $\lambda $ is an eigenvalue to the fractional boundary value problem
defined by \eqref{2.7} and \eqref{1.4}, then
\begin{equation*}
\left\vert \lambda \right\vert \geq \frac{\left( \alpha +\beta -1\right)
\Gamma \left( \alpha \right) \Gamma \left( \beta \right) }{(b-a)^{\alpha+\beta}}.  
\end{equation*}
\end{corollary}


\section{Conclusion}
\label{sec5}

We derived a new Lyapunov-type inequality for a sequential boundary 
value problem subject to natural boundary conditions. The idea of studying 
a differential equation depending on the sequence of right and left 
fractional derivatives, is relevant in applications and seems to be new.
Contrary to existing papers on Lyapunov inequalities and its generalizations, 
here the expression of the Green function $G$ is not classical and 
is expressed by integrals, which is nontrivial. 


\section*{Acknowledgments}

Guezane-Lakoud and Khaldi were supported by Algerian funds
within CNEPRU projects B01120120002 and B01120140061, 
respectively. Torres was supported by Portuguese funds 
through CIDMA and FCT, project UID/MAT/04106/2013.
The authors are grateful to an anonymous referee for 
valuable comments and suggestions, which helped to improve 
the quality of the paper.



\end{document}